\documentclass[11pt]{article}
\usepackage{amsthm}
\usepackage{graphics}
\usepackage{graphicx}
\usepackage{epsfig}
\usepackage{amssymb}
\usepackage{array}
\usepackage{mathrsfs}
\usepackage{subfigure}
\usepackage{a4wide}
\usepackage{amsmath}
\usepackage{amsfonts}
\usepackage{cite}
 \usepackage{lipsum}
\usepackage{enumerate}

\newtheorem{definition}{Definition}
\newtheorem{theorem}{Theorem}[section]
\newtheorem{lemma}{Lemma}[section]

\newtheorem{proposition}{Proposition}[section]

\newtheorem{corollary}{Corollary}[section]
\newtheorem{example}{Example}[section]

\renewcommand{\&}{\textup{\symbol{`\&}}}
\begin{document}

\title{\bf  On Dynamical System and Topological Transitivity via Ideals}
\author{Chhapikul Miah$^{a}$ and Shyamapada Modak$^{b}$ \\ $^{a}$ Sukanta Mahavidyalaya, Dhupguri, Jalpaiguri 735210, West Bengal, India.\\  $^{b}$Department of  Mathematics \\ University of Gour Banga, Malda 732 103, West Bengal, India.
India \\ E-mail:$^{a}$ chhapikul.77@gmail.com, $^{b}$ spmodak2000@yahoo.co.in,\\ Orcid Id: $^{a}$0000-0002-4965-1758, $^{b}$ 0000-0002-0226-2392 }
\date{}
\maketitle


\begin{abstract}
    \noindent
   This paper will discuss the problem of defining the new topological transitivity. To do this several equivalent topological transitivity and non-wandering point has been discussed through this paper. This paper also consider the ideal version of transitivity with the help of the amendment of the result Remark $6.9(2)$ of \cite{LL2013}.  Corrected version of the Remark: ``If $\mathcal{\bf I}$ is codense, then $\mathcal{\bf I}$-denseness, $*$-denseness and denseness are equivalent"  will be ``If $\mathcal{\bf I}$ is  completely codense, then $\mathcal{\bf I}$-denseness, $*$-denseness and denseness are equivalent".
\end{abstract}


{\bf Keywords:} transitive, ideal-transitive, complete metric space, dense orbit, $\mathcal{\bf I}$-dense orbit.\\

{\bf 2020 AMS Subject Classifications:} Primary: 54B20; Secondary: 54F15, 54C05, 54C10.


\section{Introduction}\label{sec1}

Let $X$ be a space (metric space or topological space). Then a continuous map
$f: X\rightarrow X$ together with the $X$ (i.e., (X, f)) is called a dynamical system. As we know $f\circ f (i.e., f^{2})$ is also continuous when $f$ is continuous. This point of view, the characteristics of the set $\{x_0, f(x_0), f^{2}(x_0), f^{3}(x_0), . . . ,\}$ when $x_0\in X$ is an important phenomenon. In literature, $\{x_0, f(x_0), f^{2}(x_0), f^{3}(x_0), . . . ,\}$ is called the orbit\cite{GDB1927} of $x_0$. If it is dense in $X$, then it is called dense orbit\cite{SL1997} of $x_0$. For the study of the dense property of the set $\{x_0, f(x_0), f^{2}(x_0), f^{3}(x_0), . $ $ .. ,\}$ where $x_0\in X$, $f^{n}(U)\cap V\neq\varnothing$, for any non-empty open sets $U$ and $V$ in $X$ is remarkable condition. This was introduced in literature by the name of topological transitive\cite{SL1997}. The authors Kolyada and Snoha have studied it at various angles. For this study non-wandering point\cite{LSBC1991} is another noteworthy part. For further studies, the papers \cite{AMN2009, GDB1950,GDB1927, LSBC1991, SM2011, SM2012, SW1970, RV1960} and others are relevant. In view of the above mentioned points, we are investigating the answer of the following questions through this paper:
\begin{itemize}
\item Does any set characterize the topological transitive$?$
\item For any codense ideal $\mathcal{\bf I}$, are the collection of dense sets and collection of $\mathcal{\bf I}$ dense sets coincident$?$
\item Is the iteration $\{x_0, f(x_0), f^{2}(x_0), f^{3}(x_0), . . . ,\}$  $\mathcal{\bf I}$-dense (respectively, $*$-dense)  for any $x_0\in X$ $?$
    \item Under what condition(s) non-wandering points and ideal related non-wandering points are equal$?$
\end{itemize}
To solve the above questions, we will  amend a result of \cite{LL2013}.

\section{Definitions and Notations}\label{sec2}

\begin{itemize}
  \item $\mathbb{R}$ and $\mathbb{Q}$ denote the set of reals and set of rationals respectively.
  \item $I$ be a countable set.
  \item $X$ be a topological space.
  \item $\mathcal{B}$ be a basis for the topological space $X$. Members of $\mathcal{B}$ are called basic open sets.
\item Throughout this paper, for a topological space $(X,\tau)$, we denote `$Cl$' and `$Int$' as the closure operator and interior operator respectively.
  \item If $X$ is a topological space and ${\bf A}\subseteq X$, then ${\bf A}$ is said to be semi-open\cite{NL1963} (respectivey, preopen\cite{AMN1982}, $b$-open\cite{DA1996} and $\beta$-open\cite{AEM1983}) if ${\bf A}\subseteq Cl(Int({\bf A}))$ (respectively, ${\bf A}\subseteq Int(Cl({\bf A}))$, ${\bf A}\subseteq Cl(Int({\bf A}))\cup Int(Cl({\bf A}))$ and ${\bf A}\subseteq Cl(Int(Cl({\bf A})))$). The set of all semi-open sets (respectively, preopen sets, $b$-open sets and $\beta$-open sets) in a topological space $(X,\tau)$ is denoted as $SO(X)$ (respectively, $PO(X)$, $BO(X)$ and ${\beta}O(X)$), then the relations $\tau\subseteq SO(X)\subseteq BO(X)\subseteq {\beta}O(X)$ and $\tau\subseteq PO(X)\subseteq BO(X)\subseteq {\beta}O(X)$ hold. An open set ${\bf O}$ in a topological space $X$ is said to be regular open\cite{SW1970} if ${\bf O}=Int(Cl({\bf O}))$.

 \item  A subset ${\bf A}$ of a topological space $X$ is called a dense set\cite{SW1970} in $X$ if  ${Cl(\bf A)}=X$. Equivalently ${\bf A}$ is dense in $X$ if and only if ${\bf A}$ intersects every non-empty open subset of $X$.

 \item For every pair of non-empty open sets $U$, $V$ in $X$, there exists a positive integer $n$ such that $f^{n}(U)\cap V\neq\varnothing$, then $(X, f)$ is called topological transitive \cite{SL1997}. We also say that $f$ itself
      is topologically transitive if no misunderstanding can arise.

 \item A dynamical system $(X, f)$ is called  $K$-topologically transitive if for every pair of non-empty $K$-open sets $A$ and $B$, there exists a positive integer $n$ such that $f^{n}(A)\cap B\neq\varnothing$. In this case, $K$ may be the semi-open set \cite{NL1963}, preopen set \cite{AMN1982}, $b$-open set \cite{DA1996}, $\beta$-open set \cite{AEM1983}.

 \item A point $x\in X$ is said to be non-wandering if every neighbourhood $U$ of $x$ there exists a positive integer $n$  such that $f^{n}(U)\cap U\neq\varnothing$ \cite{LSBC1991}. The collection of all non-wandering points of $f$ will be denoted by $\Omega_{f}(x)$.

  \item A point $x\in X$ is said to be semi (respectively, pre, $b$, $\beta$) non-wandering  if every semi(respectively, pre, $b$, $\beta$) open set $U$ of $x$ there exists a positive integer $n$  such that $f^{n}(U)\cap U\neq\varnothing$.

  \item A point $x\in X$ ``moves,'' its trajectory being the sequence $x, f(x), f^{2}(x), f^{3}(x), . . . ,$ where $f^{n}$ is the $n^{th}$ iteration of $f$ . The point $f^{n}(x)$ is the position of $x$ after $n$ units of time. The set of points of the trajectory of $x$ under $f$ is called the orbit of $x$, denoted by $orb_{f}(x)$(or simply $(O_{f}(x))$)\cite{SL1997}.
  \item A dynamical system $(X, f)$ has a dense orbit if there exists a point $x\in X$ such that the orbit $\{x, f(x), f^{2}(x),f^{3}(x),...\}$ is dense in $X$ \cite{SL1997}.
   \item Any point with dense orbit is called a transitive point. Any point which is not transitive is intransitive \cite{SL1997}. The set of transitive or intransitive points of $(X, f)$ will be denoted by $tr(f)$ or $intr(f)$ respectively.

 \item Ideal: According to Kuratowski \cite{KK1966}, a collection $\mathcal{\bf I}\subseteq 2^{X}$ is called an ideal on $X$ if $\mathcal{\bf I}$ is closed under hereditary property and finite additivity property. If $X\notin\mathcal{\bf I}$, then $\mathcal{\bf I}$ is called proper ideal. Due to the ideal $\mathcal{\bf I}$ on a topological space $(X, \tau)$, Kuratowski's local function is, $A^{*}=\{x\in X:U_{x}\cap A\notin \mathcal{\bf I} $ $\forall U_{x}\in \tau(x)\}$, where $\tau(x)=\{U\in\tau: x\in U\}$. Its associated set valued set function \cite{SS2021} (or complementary operator is $\psi$ \cite{TN1986} and it is defined as, $\psi(A)=X\setminus(X\setminus A)^{*}$ \cite{TN1986}, we will write simply $\psi_{\tau}(M)$ or $\psi_{\tau}^{\mathcal{\bf I}}(M)$). $Cl^{*}(A)=A\cup A^{*}$ determines a Kuratowski closure operator for a topology $\tau^{*}(\mathcal{\bf I})=\{A\subset X:Cl^{*}(X\setminus A)=X\setminus A\}$ \cite{JH1990} or simply $\tau^{*}$, called the $*$-topology, which is finer than $\tau$. The member of $\tau^{*}(\mathcal{\bf I})$ is called $*$-open set and the compliment of $*$-open set is called $*$-closed set. $A\subset X$ is $*$-open if and only if $Int^*(A)=A$ where $Int^*(A)$ is denoted interior of $A$ with respect to $*$-topology and $A\subset X$ is $*$-closed if and only if $Cl^*(A)=A$ where $Cl^*(A)$ is denoted closure of $A$ with respect to $*$-topology.  A basis for the $*$-topology $\tau^{*}(\mathcal{\bf I})$ is ${\bf {\beta}}=\{O\setminus I_1:O\in\tau, I_1\in \mathcal{\bf I}\}$\cite{JH1990}.
\item For the condition $\tau\cap \mathcal{\bf I}=\{\varnothing\}$ of an ideal topological space $(X, \tau, \mathcal{\bf I})$, we say that $\mathcal{\bf I}$ is $\tau$-boundary\cite{RLN1967}(or codense\cite{DGR1999}) and this type of space is called Hayashi-Samuel space \cite{DJ1999}.
 \item According to Nj{\aa}stad, the ideal $\mathcal{\bf I}$ is said to be compatible with $\tau$, denoted by $\mathcal{\bf I}\sim \tau$ \cite{ON1966} if the following holds for every $A\subseteq X$: if for all $x\in A$, there exists $O\in\tau(x)$  where $\tau(x)=\{O\in \tau:x\in U\}$ such that $O\cap A\in \mathcal{\bf I}$, then $A\in \mathcal{\bf I}$.
\item For $\mathcal{\bf I}\sim \tau$, $\psi(\psi(A))=\psi(A)$ \cite{HJ1990} and $\psi(A)\setminus A\in \mathcal{\bf I}$  \cite{HJ1990} for every $A\subseteq X$.

  \item A subset $\bf A$ of a topological space $(X, \tau)$ with an ideal $\mathcal{\bf I}$ is called $\mathcal{\bf I}$-dense if and only if ${\bf A}^{*}(\mathcal{\bf I})=X$ \cite{DGR1999}. Equivalently ${\bf A}$ is $\mathcal{\bf I}$-dense in $X$ if and only if for every non-empty open subset ${\bf O}$ of $X$, ${\bf A}\cap {\bf O}\notin \mathcal{\bf I}$. Every $\mathcal{\bf I}$-dense subset of a topological space is a dense subset.  A subset $\bf A$ of a topological space $(X, \tau)$ with an ideal $\mathcal{\bf I}$ is  $\mathcal{\bf I}$-dense if and only if $\psi(X\setminus A)=\varnothing$.

  \item A subset $\bf A$ of a topological space $(X, \tau)$ with an ideal $\mathcal{\bf I}$ is called $*$-dense if and only if ${Cl}^{*}(A)=X$ \cite{EH1964}. Equivalently ${\bf A}$ is $*$-dense in $X$ if and only if for every non-empty $*$-open subset ${\bf O}$ of $X$, ${\bf A}\cap {\bf O}\neq\varnothing$. Every $\mathcal{\bf I}$-dense subset of $X$ is a $*$-dense but reverse may not be true\cite{LL2013}.
   \item  An ideal $\mathcal{\bf I}$ on a topological space is called a completely codense if $PO(X,\tau)\cap\mathcal{\bf I}=\{\varnothing\}$\cite{DGR1999}.
 \item An ideal $\mathcal{\bf I}$ is completely codense on $(X,\tau)$ if and only if each member of $\mathcal{\bf I}$ is nowhere dense \cite{DGR1999}.
  \item A subset $\bf A$ of a topological space $(X, \tau)$ with an ideal $\mathcal{\bf I}$ is called $\mathcal{\bf I}$-open if and only if ${\bf A}\subset Int({\bf A}^*)$. The collection of all  $\mathcal{\bf I}$-open\cite{AEA1992} sets in an ideal topological space $(X, \tau,\mathcal{\bf I})$ is denoted by $\mathcal{\bf I}O(X)$.
   \item Let $f:X \rightarrow Y$ be a function. If $\mathcal{\bf I}$ is an ideal on $X$, then $f(\mathcal{\bf I})=\{f(\mathcal{I}_{1}):\mathcal{I}_{1}\in \mathcal{\bf I} \}$ is also an ideal on $Y$ \cite{DJ1992}. Moreover $f^{\leftarrow}(\mathcal{\bf I})=\{A:A\subset f^{-1}(I)\in f^{-1}(\mathcal{\bf I})\}$ is an ideal on $X$\cite{JS2023}.

    \item $A\subseteq  X$ is invariant if $f(A)\subseteq A$ \cite{SL1997}.
    \item Throughout this paper we will write $\mathcal{\bf I }$-space instead of  ideal topological space and space instead of topological space if no misunderstanding can arise.
\end{itemize}

\section{Equivalent definition of topological  transitivity}\label{sec3}

This section will discuss that which collections of a topological space are not suitable for defining new topological transitive. That is, we discuss equivalent definition of topological transitivity.\\

\begin{lemma}
\label{lem.tct.3.1}
Let $(X, f)$ be a dynamical system and $\mathcal{B}$ be a basis of the topological space $X$. Then the following are equivalent statements:\\
$(i)$ $f$ is topological transitive;\\
$(ii)$ for every pair of non-empty open sets $A$ and $B$, there exists a positive integer $n$ such that $Cl(f^{n}(A))\cap B\neq \varnothing$;\\
$(iii)$ for every pair of basic open sets $B_{1}, B_{2}\in \mathcal{B}$, there exists a positive integer $n$ such that $f^{n}(B_{1})\cap B_{2}\neq \varnothing$.
\end{lemma}

\begin{lemma}
\label{lem.tct.3.2}
Let $(X, f)$ be a dynamical system. Then the following statements are equivalent:\\
$(i)$  $f$ is topological semi (respectively, pre, $b$, $\beta$) transitive;\\
$(ii)$ for every pair of non-empty semi-open (respectively, preopen, $b$-open, $\beta$-open) sets $A$ and $B$, there exists a positive integer $n$ such that $ scl(f^{n}(A))\cap B ( \text{respectively}, pcl(f^{n}(A))\cap B, bcl(f^{n}(A))\cap B, \beta-cl(f^{n}(A))\cap B)\neq \varnothing$.
\end{lemma}

\begin{lemma}
\label{lem.tct.3.3}
Let $(X, f)$ be a dynamical system and $\mathcal{B}$ be a basis of the topological space $X$. Then the following statements are equivalent:\\
$(i)$ $x$ is a non-wandering point of $X$;\\
$(ii)$ for every neighbourhood $M$ of  $x$, there exists a positive integer $n$ such that $Cl(f^{n}(M))\cap M\neq \varnothing$;\\
$(iii)$ for every  basic element  $B$ of $\mathcal{B}$ containing $x$, there exists a positive integer $n$ such that $Cl(f^{n}(B))\cap B\neq \varnothing$.
\end{lemma}
\begin{lemma}

\label{lem.tct.3.4}
Let $(X, f)$ be a dynamical system. Then the following statements are equivalent:\\
$(i)$ $x$ is a semi (respectively, pre, $b$, $\beta$) non-wandering point of $X$;\\
$(ii)$ for every semi-open (respectively, preopen, $b$-open, $\beta$-open)set $M$ containing  $x$, there exists a positive integer $n$ such that $scl(f^{n}(M))\cap M (\text{respectively}, pcl(f^{n}(M))\cap M, bcl(f^{n}(M))\cap M, \beta-cl(f^{n}(M))\cap M)\neq \varnothing$.
\end{lemma}

\section{Ideal version of topological transitive}\label{sec4}
In this section, we have introduced ideal version of transitivity called topologically ideal transitive and investigate its various characterizations and features. To do this we will modify the Remark $6.9(2)$ of \cite{LL2013}.
\begin{definition}
  \label{def.ncf.4.1}
   Let $(X,\tau,\mathcal{\bf I})$ be an ideal topological space. A dynamical system $(X, f)$  is called a topologically ideal transitive (or simply ideal transitive or $\mathcal{\bf I}$-transitive) if for every pair of non-empty open sets $A$ and $B$, there exists a positive integer $n$ such that $f^{n}(A)\cap B\notin \mathcal{\bf I}$.
\end{definition}

\begin{example}
  \label{exm.tit.4.0}
  Suppose $X=\mathbb{R}$ (set of all real numbers) is a set endowed with the co-finite topology $\tau$ and  $\mathcal{\bf I}=\mathcal{\bf I_c}$, ideal of all countable subsets of $\mathbb{R}$. Let us define a mapping $f:X \rightarrow X$ by $f(x)=x$, for all $x\in \mathbb{R}$. Then, $f$ is $\mathcal{\bf I}$-transitive.
\end{example}
\begin{example}
  \label{exm.tct.4.0.1}
  Suppose $X=\{a, b, c, d\}$ is a set endowed with the  topology $\tau=\{\varnothing, X, \{a\}, \{c\}, $ $\{a, c\}, \{a, c, d\}\}$ and $\mathcal{\bf I}=\{\varnothing, \{d\}\}$. Let us define a mapping $f:X \rightarrow X$ by $f(a)=c$, $f(b)=b$, $f(c)=a$ and  $f(d)=d$. Then $f$ is  $\mathcal{\bf I}$-transitive.
\end{example}
Every $\mathcal{\bf I}$-transitive is topologically transitive. But the converse may not be true. We consider the following examples:

\begin{example}
  \label{exm.tct.4.1}
  Suppose $X=\{a, b, c, d\}$ is a set endowed with the  topology $\tau=\{\varnothing, X, \{a\}, \{c\}, $ $\{a, c\}, \{a, c, d\}\}$ and $\mathcal{\bf I}=\{\varnothing, \{a\}, \{c\}, \{a, c\}\}$. Let us define a mapping $f:X \rightarrow X$ by $f(a)=c$, $f(b)=b$, $f(c)=a$ and  $f(d)=d$. Then $f$ is open, continuous and topologically transitive but not $\mathcal{\bf I}$-transitive.
\end{example}
\begin{example}
  \label{exm.tit.4.2}
  Suppose $X=\mathbb{R}$  is a set endowed with the  topology $\tau=\{\varnothing, \mathbb{R},(a,b), [a,b), [a,b]\}$ with $a< b$ and $a,b \in\mathbb{R}$ and $\mathcal{\bf I}=\wp([a,b])$ (power set of $[a, b]$). Let us define a mapping $f:X \rightarrow X$ by $f(x)=x$, for all $x\in \mathbb{R}$. Then, $f$ is topologically transitive but not $\mathcal{\bf I}$-transitive.
\end{example}

If $\mathcal{I}=\{\varnothing\}$, then  $\mathcal{\bf I}$-transitive and topologically transitive becomes identical.\\
For the reverse inclusion we consider the followings:
\begin{theorem}
  \label{thm.tit.4.1}
  Let $\mathcal{\bf I}$ be a codense ideal on a topological space $(X, \tau)$  and $(X, f)$ be a dynamical system. If $f$ is open and topologically  transitive, then it is $\mathcal{\bf I}$-transitive.
\end{theorem}
\begin{proof}
  By topologically transitivity of $f$ implies for every non-empty open sets $U$ and $V$, there exists a positive integer $k$ such that $f^{k}(U)\cap V\neq\varnothing$. Since $f$ is an open mapping, then $f^{k}(U)\cap V\in\tau$. Hence, by codenseness of $\mathcal{\bf I}$, $f^{k}(U)\cap V\notin\mathcal{\bf I}$. Thus, $f$ is $\mathcal{\bf I}$-transitive. This completes the proof.
\end{proof}
\begin{theorem}
  \label{thm.tit.4.7}
  Let $(X, \tau, \mathcal{\bf I})$ be an ideal topological space. If $(X,f)$ is $\mathcal{\bf I}$-transitive, then the space $(X, \tau, \mathcal{\bf I})$ is a Hayashi-Samuel space.
\end{theorem}
\begin{proof}
  Let $U(\neq\varnothing)$ be an open set in $X$. Then, for $\mathcal{\bf I}$-transitive of $(X,f)$, for any $V(\neq\varnothing)\in\tau$, there exists a positive integer $n$ such that $f^{n}(V)\cap U\notin\mathcal{\bf I}$. This implies $U\notin \mathcal{\bf I}$ (if $U\in\mathcal{\bf I}$, then $f^{n}(V)\cap U\in\mathcal{\bf I}$) and hence $(X, \tau, \mathcal{\bf I})$ is a Hayashi-Samuel space. This completes the proof.
\end{proof}

But the converse may not be true. we consider the following examples:
\begin{example}
  \label{exm.tit.4.3}
  Suppose $X=\mathbb{R}$  is a set endowed with the usual topology $\tau_{u}$  and consider an ideal $\mathcal{\bf I}=\mathcal{P}(\mathbb{R}\setminus\mathbb{Q})$ (power set of $\mathbb{R}\setminus\mathbb{Q}$). Let us define a mapping $f:X \rightarrow X$ by $f(x)=x$, for all $x\in \mathbb{R}$. If we take, $U=(0,1)$ and $V=(1,2)$ be two open subsets of $X$, then for any positive integer $n$, $f^{n}(U)\cap V\in\mathcal{\bf I}$.  Thus, $f$ is not $\mathcal{\bf I}$-transitive but the space $(X, \tau, \mathcal{\bf I})$ is a Hayashi-Samuel space.
\end{example}
\begin{example}
  \label{exm.tit.4.4}
  Suppose $X= \{0\}\cup\{\frac{1}{n} : n \in \mathbb{N}\}$ (where $\mathbb{N}$ denotes the set of all natural numbers) is a set endowed with the usual topology and consider an ideal $\mathcal{\bf I}=\{\varnothing, \{0\}\}$. Let us define a mapping $f:X \rightarrow X$ by $f(\frac{1}{n})=\frac{1}{(n+1)}$, $n=1, 2, 3,...$ and $f(0)=0$. Clearly $f$ is open and continuous. Take, $U=\{\frac{1}{2}\}, V=\{1\}$ be two open subsets of $X$. Then, for all $n\in \mathbb{N}$, $f^{n}(U)\cap V\in \mathcal{\bf I}$.  Thus, $f$ is not $\mathcal{\bf I}$-transitive but the space $(X, \tau, \mathcal{\bf I})$ is a Hayashi-Samuel space.
\end{example}

We learnt from various research papers that dense sets are important for discussing the properties of topological transitive. In this point of view to discuss the properties of ideal version of topological transitive, $\mathcal{\bf I}$-dense sets play an important role. Thus we mentioned that if the collection of dense sets and the collection of $\mathcal{\bf I}$-dense sets are equal, then the properties of $\mathcal{\bf I}$-transitive will not be interesting. Therefore it is necessary to amend the Remark $6.9(2)$ of \cite{LL2013},  ``{\bf $\mathcal{\bf I}$-dense, $*$-dense and dense are  equivalent if  the ideal $\mathcal{\bf I}$ is codense}". Followings are the  counter examples to show that the Remark $\bf 6.9(2)$ of {\bf \cite{LL2013}} is not true always. For this we consider the following examples:
\begin{example}
  Let us consider an ideal topological space $(X, \tau, \mathcal{\bf I})$ where $X=\{a, b, c, d\}$, $\tau=\{\varnothing, X, \{a, c, d\}\}$ and $\mathcal{\bf I}=\{\varnothing, \{a\}, \{c\}, \{a,c\} \}$.  Here $\mathcal{\bf I}\cap \tau=\{\varnothing\}$ and hence $\mathcal{\bf I}$ is codense. Take $A=\{a, b, c\}$. Then, $Cl(A)=X$. This implies $A$ is a dense subset of $X$. Also $A^{*}=\{b, c, d\}\neq X$ and $Cl^{*}(A)=A\cup A^{*}=\{a,b,c,d\}$. This shows that $A$ is $*$-dense but not $\mathcal{\bf I}$-dense. Thus $\mathcal{\bf I}$-dense, $*$ and dense are not equivalent though $\mathcal{\bf I}$ is codense.
\end{example}
\begin{example}
  \label{exm.tit.4.0.1}
  Suppose $X=\mathbb{R}$  is a set endowed with the co-finite topology $\tau$ and  $\mathcal{\bf I}=\mathcal{\bf I_c}$, ideal of all countable subsets of $\mathbb{R}$. Then $\mathcal{\bf I}\cap \tau=\{\varnothing\}$ and hence $\mathcal{\bf I}$ is codense. Take $A=\mathbb{Q}$. Then, $Cl(A)=\mathbb{R}$. This implies $A$ is a dense subset of $X$. Also $A^{*}=\varnothing\neq X$ and hence $Cl^*{A}=A\cup A^{*}=A\neq X$. This shows that $A$ is neither $\mathcal{\bf I}$-dense nor $*$-dense. Hence $\mathcal{\bf I}$-dense, $*$-dense and dense are not equivalent though $\mathcal{\bf I}$ is codense.
\end{example}
Above examples showed that `` $\mathcal{\bf I}$-dense, $*$-dense and dense are not equivalent" when $\mathcal{\bf I}$ is codense.\\
 Even if $\mathcal{\bf I}$ is codense and $\mathcal{\bf I}\sim\tau$, by the following example, we have shown that  $\mathcal{\bf I}$-dense, $*$-dense and dense are not equivalent:

\begin{example}
  \label{exm.tit.4.0.3}
  Suppose $X=\mathbb{R}$  is a set endowed with the topology $\tau=\{\varnothing, \mathbb{R}, \mathbb{Q}\cup \{i\}\}$ for some fixed $i\in \mathbb{R}\setminus \mathbb{Q}$ and  $\mathcal{\bf I}=\mathcal{P}(\mathbb{Q})$ (power set of $\mathbb{Q}$). Then $\mathcal{\bf I}\cap \tau=\{\varnothing\}$ and hence $\mathcal{\bf I}$ is codense. Also $\mathcal{\bf I}\sim \tau$.  Take $A=\mathbb{Q}$. Then, $Cl(A)=\mathbb{R}$. This implies $A$ is a dense subset of $\mathbb{R}$. Also $A^{*}=\varnothing\neq \mathbb{R}$ and hence $Cl^*{A}=A\cup A^{*}=A\neq \mathbb{R}$. This shows that $A$ is neither $\mathcal{\bf I}$-dense nor $*$-dense. Hence $\mathcal{\bf I}$-dense, $*$-dense and dense are not equivalent though $\mathcal{\bf I}$ is codense and $\mathcal{\bf I}\sim \tau$.
\end{example}
In this regards, we shall give a modification of the Remark $\bf 6.9(2)$ of {\bf \cite{LL2013}} in the next section.

\begin{theorem}
  \label{thm.tit.4.2}
  Let $X$ be an $\mathcal{\bf I}$-space and $(X, f)$ be a dynamical system. If $f$ is $\mathcal{\bf I}$-transitive. Then, for any non-empty open set $U$ in $X$, $\bigcup\limits_{n=0}^{\infty}f^{n}(U)$ is $\mathcal{\bf I}$-dense in $X$.
\end{theorem}

\begin{proof}
Let $f$ be an $\mathcal{\bf I}$-transitive map and assume that $\bigcup\limits_{n=0}^{\infty}f^{n}(U)$ is not $\mathcal{\bf I}$-dense in $X$. Then, there exists a non-empty open set $V$ in $X$ such that $(\bigcup\limits_{n=0}^{\infty}f^{n}(U)) \cap V\in \mathcal{\bf I}$.
This implies $f^{n}(U)\cap V\in\mathcal{\bf I}$ (by hereditary property of $\mathcal{\bf I}$) for all $n\in\mathbb{Z_{+}}\cup \{0\}$ where $\mathbb{Z_{+}}$ is the set of all positive integers. This leads a contradiction to the $\mathcal{\bf I}$-transitive of $f$. Hence $\bigcup\limits_{n=0}^{\infty}f^{n}(U)$ is $\mathcal{\bf I}$-dense in $X$. This completes the proof.
\end{proof}

But the reverse may not be true. For this, let $U$ and $V$ be two non-empty open sets in $X$ and assume that
 $\bigcup\limits_{n=0}^{\infty}f^{n}(U)$ is $\mathcal{\bf I}$-dense in $X$. Then, $\bigcup\limits_{n=0}^{\infty}f^{n}(U)\cap V\notin\mathcal{\bf I}$ i.e., $\bigcup\limits_{n=0}^{\infty}(f^{n}(U)\cap V)\notin\mathcal{\bf I}$ . This implies for all $m\in \mathbb{Z_{+}}\cup \{0\}$, $f^{m}(U)\cap V$ may  belongs to $\mathcal{\bf I}$. \\
 If we consider, $\mathcal{\bf I}$ of all proper subsets of $\mathbb{N}$ where $\mathbb{N}$ denotes the set of all natural numbers, then $\mathcal{\bf I}$ becomes an ideal. Here for all $n\in \mathbb{N}$, $\{n\}\in\mathcal{\bf I}$ but $\bigcup\limits_{n=1}^{\infty}\{n\}\notin\mathcal{\bf I}$.\\

We know that $\mathcal{\bf I}$-transitive implies $\mathcal{\bf I}$ is codense. Again in Hayashi-Samuel space, the collection of $\mathcal{\bf I}$-dense sets and the collection of dense sets in the $*$-topology are equal\cite{SM2007}. Therefore we get the following:
\begin{corollary}
  \label{cor.tit.4.1.0}
  Let $X$ be an $\mathcal{\bf I}$-space and $(X, f)$ be a dynamical system. If $f$ is $\mathcal{\bf I}$-transitive. Then, for any non-empty open set $U$ in $X$, $\bigcup\limits_{n=0}^{\infty}f^{n}(U)$ is $*$-dense in $X$.
\end{corollary}

\begin{theorem}
  \label{thm.tit.4.3}
  Let $X$ be an $\mathcal{\bf I}$-space and $(X, f)$ be a dynamical system. Then $f$ is $\mathcal{\bf I}$-transitive  if for any non-empty open set $U$ in $X$, $\bigcup\limits_{n=0}^{\infty}f^{n}(U)$ is $\mathcal{\bf I}$-dense in $X$ and $\mathcal{\bf I}$ is countably additive.
\end{theorem}
\begin{proof}
  let $U$ and $V$ be two non-empty open sets in $X$ and assume that
 $\bigcup\limits_{n=0}^{\infty}f^{n}(U)$ is $\mathcal{\bf I}$-dense in $X$. Then, $\bigcup\limits_{n=0}^{\infty}f^{n}(U)\cap V\notin\mathcal{\bf I}$. This implies there exists $m\in \mathbb{Z_{+}}$ such that $f^{m}(U)\cap V\notin\mathcal{\bf I}$, since $\mathcal{\bf I}$ is countably additive. Hence, $f$ is $\mathcal{\bf I}$-transitive. This completes the proof.
\end{proof}

  \begin{corollary}
  \label{cor.tit.4.1.1}
  Let $X$ be an $\mathcal{\bf I}$-space and $(X, f)$ be a dynamical system. Then $f$ is $\mathcal{\bf I}$-transitive  if for any non-empty open set $U$ in $X$, $\bigcup\limits_{n=0}^{\infty}f^{n}(U)$ is $*$-dense in $X$ and $\mathcal{\bf I}$ is countably additive.
\end{corollary}

\begin{theorem}
  \label{thm.tit.4.4}
  Let $X$ be an $\mathcal{\bf I}$-space  and $(X, f)$ be a dynamical system. If $f$ is $\mathcal{\bf I}$-transitive, then for any non-empty open set $U$ in $X$, $\bigcup\limits_{n=0}^{\infty}f^{-n}(U)$ is $\mathcal{\bf I}$-dense in $X$.
\end{theorem}
\begin{proof}
  Assume that $f$ is $\mathcal{\bf I}$-transitive.
Since $f$ is a continuous function and arbitrary union of open sets is open in a topological space, then $\bigcup\limits_{n=0}^{\infty}f^{-n}(U)$ is open. Again, since $f$ is $\mathcal{\bf I}$-transitive, $\bigcup\limits_{n=0}^{\infty}f^{-n}(U)$ is $\mathcal{\bf I}$-dense in $X$. This completes the proof.
\end{proof}

\begin{corollary}
  \label{cor.tit.4.1.2}
  Let $X$ be an $\mathcal{\bf I}$-space  and $(X, f)$ be a dynamical system. If $f$ is $\mathcal{\bf I}$-transitive, then for any non-empty open set $U$ in $X$, $\bigcup\limits_{n=0}^{\infty}f^{-n}(U)$ is $*$-dense in $X$.
\end{corollary}

\begin{theorem}
  \label{thm.tit.4.5}
  Let $X$ be an $\mathcal{\bf I}$-space  and $(X, f)$ be a dynamical system. Then $f$ is topologically $f^{k}(\mathcal{\bf I})$-transitive for some positive integer $k$ if  for every non-empty open set $U$ in $X$, $\bigcup\limits_{n=0}^{\infty}f^{-n}(U)$  is $\mathcal{\bf I}$-dense in $X$  and $\mathcal{\bf I}$ is countably additive.
\end{theorem}
\begin{proof}
  Let $U$ and $V$ be any two non-empty open sets in $X$. Since, $\bigcup\limits_{n=0}^{\infty}f^{-n}(U)$ is $\mathcal{\bf I}$-dense in $X$, then $\bigcup\limits_{n=0}^{\infty}f^{-n}(U)\cap V\notin\mathcal{\bf I}$ for every non-empty open sets $V$ in $X$. This implies there exists a positive integer $k$ such that $f^{-k}(U)\cap V\notin\mathcal{\bf I}$ since $\mathcal{\bf I}$ is countably additive. This implies, there exists a positive integer $k$ such that $U\cap f^{k}(V)\notin f^{k}(\mathcal{\bf I})$. Hence, $f$ is topologically $f^{k}(\mathcal{\bf I})$-transitive. This completes the proof.
\end{proof}

 \begin{corollary}
  \label{cor.tit.4.1.3}
  Let $X$ be an $\mathcal{\bf I}$-space  and $(X, f)$ be a dynamical system. Then $f$ is topologically $f^{k}(\mathcal{\bf I})$-transitive for some positive integer $k$ if  for any non-empty open set $U$ in $X$, $\bigcup\limits_{n=0}^{\infty}f^{-n}(U)$  is $*$-dense in $X$  and $\mathcal{\bf I}$ is countably additive.
\end{corollary}

\begin{theorem}
  \label{thm.tit.4.5}
  Let $X$ be an $\mathcal{\bf I}$-space and $(X, f)$ be a dynamical system and $\mathcal{\bf I}$-transitive. If $A\subseteq X$ is closed and $f(X\setminus A)\subseteq X \setminus A$, then either $A=X$ or $\psi(A^{*}(\mathcal{\bf I}))= \varnothing$.
\end{theorem}
\begin{proof}
  Let $f$ is $\mathcal{\bf I}$-transitive, $A\subseteq X$ is closed and $f(X\setminus A)\subseteq X\setminus A$. Assume that $A\neq X$ and  $\psi(A^{*}(\mathcal{\bf I}))\neq \varnothing$. Put $U=X\setminus A$. Then $U$ is open, and $f^{n}(U)\subseteq U$. Since $\psi(A^{*}(\mathcal{\bf I}))\neq \varnothing$, then there exists an open set $V$ in $X$ such that $V \setminus A^{*}(\mathcal{\bf I})\in \mathcal{\bf I}$ i.e. $V \cap (X\setminus A^{*}(\mathcal{\bf I}))\in \mathcal{\bf I}$. Also for every subset $A$ of $X$, $A^{*}(\mathcal{\bf I})\subseteq Cl(A)$. Since $A$ is closed, then $A^{*}(\mathcal{\bf I})\subseteq Cl(A)=A$ and hence $V \cap (X\setminus A)\in \mathcal{\bf I}$ (by hereditary property of $\mathcal{\bf I}$). Further, we have, $U=X\setminus A$ and $f^{n}(U)\subseteq U$ for  all $n\in\mathbb{Z_{+}}$. This implies, $V \cap f^{n}(U)\in \mathcal{\bf I}$ for  all $n\in\mathbb{Z_{+}}$ which contradicts the fact that $f$ is $\mathcal{\bf I}$-transitive. Hence either $A=X$ or  $\psi(A^{*}(\mathcal{\bf I}))=\varnothing$. This completes the proof.
\end{proof}

\begin{theorem}
  \label{thm.tit.4.6}
  Let $X$ be an $\mathcal{\bf I}$-space. The following statements are equivalent:\\
  $(1)$ $f$ is $\mathcal{\bf I}$-transitive.\\
   $(2)$ if $U\subseteq X$ is open and $f(U)=U$, then either $U=\varnothing$ or $U$ is $\mathcal{\bf I}$-dense in $X$.\\
   $(3)$ if $U\subseteq X$ is open and $f(U)=U$, then either $U=\varnothing$ or $U$ is $*$-dense in $X$.
\end{theorem}
\begin{proof}
  $(1)\implies (2)$ Let $f$ is $\mathcal{\bf I}$-transitive, $U\subseteq X$ is open and $f(U)=U$. Assume that $U\neq \varnothing$ and $U$ is not $\mathcal{\bf I}$-dense in $X$. Then there exists a non-empty open set $V$ in $X$ such that $U\cap V\in \mathcal{\bf I}$. Since  $f(U)=U$, then $f^{n}(U)=U$  for all $n\in\mathbb{Z_{+}}$ and hence $f^{n}(U)\cap V\in \mathcal{\bf I}$ for all $n\in\mathbb{Z_{+}}$ which contradicts the ideal transitivity of $f$. Hence either $U=\varnothing$ or $U$ is $\mathcal{\bf I}$-dense in $X$ in $X$.\\


$(2)\implies (1)$ Assume that the condition $(2)$ holds. Let $U$ be non-empty open subset of $X$ with $f(U)=U$. Then, $U$ is $\mathcal{\bf I}$-dense in $X$. Then, for all non-empty open subsets $V$ in $X$, $U\cap V\notin \mathcal{\bf I}$. Assume that $f$ is not $\mathcal{\bf I}$-transitive, then there exists a non-empty open set $V_{1}$ in $X$ such that $f^{n}(U)\cap V_{1}\in \mathcal{\bf I}$ for all $n\in\mathbb{Z_{+}}$. This implies, $U\cap V_{1}\in \mathcal{\bf I}$ for all $n\in\mathbb{Z_{+}}$  which contradicts the fact that $U\cap V\notin \mathcal{\bf I}$ for all non-empty open subsets $V$ in $X$. Hence $f$ is $\mathcal{\bf I}$-transitive.\\
This completes the proof.
\end{proof}

\begin{corollary}
  \label{cor.tit.4.1}
  Let $X$ be an $\mathcal{\bf I}$-space and $f$ be $\mathcal{\bf I}$-transitive. If $U\subseteq X$ is open and invariant, then either $U=\varnothing$ or $U$ is $\mathcal{\bf I}$-dense (respectively, $*$-dense) in $X$.
\end{corollary}
Recall the following results:
\begin{theorem}[\cite{JH1990}]
  \label{thm.tct.4.8.1}
  Let $(X, \tau, \mathcal{\bf I})$ be an $\mathcal{\bf I}$-space. Then the following statements are equivalent:\\
  $(1)$ $\tau$-boundary.\\
  $(2)$ $\psi(\varnothing)=\varnothing$.\\
  $(3)$ $X=X^{*}$.\\
  $(4)$ For every $A\in\tau$, $A\subset A^{*}$.\\
  $(4)$ if $I\in\mathcal{\bf I}$, then $\psi(I)=\varnothing.$
\end{theorem}
\begin{theorem}[\cite{HJ1990}]
  \label{thm.tct.4.8}
  Let $(X, \tau, \mathcal{\bf I})$ be an $\mathcal{\bf I}$-space. Then the following statements are equivalent:\\
  $(1)$ $\mathcal{\bf I}\sim \tau$\\
  $(2)$ For every $A\subseteq X$, $A\setminus A^{*}\in\mathcal{\bf I}$
\end{theorem}

\begin{theorem}
  \label{thm.tit.4.8}
  Let $(X, \tau, \mathcal{\bf I})$ be an $\mathcal{\bf I}$-space with $\mathcal{\bf I}\sim \tau$ and $(X, f)$ is a dynamical system.  If $f$ is $\mathcal{\bf I}$-transitive, then $X\setminus(A\setminus {A^*})$ is  $\mathcal{\bf I}$-dense (respectively, $*$-dense) for any $A\subset X$.
\end{theorem}
\begin{proof}
 Let $f$ is $\mathcal{\bf I}$-transitive. Then by Theorem \ref{thm.tit.4.7}, the ideal $\mathcal{\bf I}$ is codense i.e., $\mathcal{\bf I}$ is $\tau$-boundary . Since $\mathcal{\bf I}\sim \tau$, then by Theorem \ref{thm.tct.4.8}, for every $A\subseteq X$, $A\setminus A^{*}\in\mathcal{\bf I}$. Also since $\mathcal{\bf I}$ is $\tau$-boundary, then by Theorem \ref{thm.tct.4.8.1}, $\psi(A\setminus A^{*})=\varnothing$. This implies, $X\setminus(X\setminus(A\setminus A^{*}))^{*}=\varnothing$ and hence $X=(X\setminus(A\setminus A^{*}))^{*}$. Thus $X\setminus(A\setminus A^{*})$ is $\mathcal{\bf I}$-dense for every $A\subset X$. This completes the proof.
 \end{proof}

\begin{theorem}
   Let $(X, \tau, \mathcal{\bf I})$ be an $\mathcal{\bf I}$-space with $\mathcal{\bf I}\sim \tau$ and $(X, f)$ is a dynamical system.  If $f$ is $\mathcal{\bf I}$-transitive, then $X\setminus(\psi(A)\setminus A)$ is  $\mathcal{\bf I}$-dense (respectively, $*$-dense) for any $A\subset X$.
\end{theorem}
 \begin{proof}
  Let $f$ is $\mathcal{\bf I}$-transitive. Then by Theorem \ref{thm.tit.4.7}, the ideal $\mathcal{\bf I}$ is codense i.e., $\mathcal{\bf I}$ is $\tau$-boundary. Since $\mathcal{\bf I}\sim \tau$, then $\psi(A)\setminus A\in \mathcal{\bf I}$  \cite{HJ1990} for every $A\subseteq X$. Since $\mathcal{\bf I}$ is $\tau$-boundary, $\psi(\psi(A)\setminus A)=\varnothing$. This implies, $X\setminus(\psi(A)\setminus A)$ is  $\mathcal{\bf I}$-dense for any $A\subset X$. This completes the proof.
 \end{proof}
\begin{theorem}
\label{thm.tit.4.8.2}
   Let $(X, \tau, \mathcal{\bf I})$ be an $\mathcal{\bf I}$-space and $(X, f)$ be a dynamical system.  If $f$ is $\mathcal{\bf I}$-transitive, then $X\setminus A^*\subset (X\setminus A)^*$ for any $A\subset X$.
\end{theorem}
\begin{proof}
  Let $f$ is $\mathcal{\bf I}$-transitive. Then by Theorem \ref{thm.tit.4.7}, the ideal $\mathcal{\bf I}$ is codense i.e., $\mathcal{\bf I}$ is $\tau$-boundary. Hence $X=X^*$ by Theorem \ref{thm.tct.4.8.1}. Now, $A^*\cup (X\setminus A)^*=(A\cup (X\setminus A))^*=X^*=X$ \cite{JH1990} and hence $X\setminus A^*\subset (X\setminus A)^*$. Thus, $X\setminus A^*\subset (X\setminus A)^*$ for every $A\subset X$. This, completes the proof.
  \end{proof}
\begin{theorem}
   Let $(X, \tau, \mathcal{\bf I})$ be an $\mathcal{\bf I}$-space and $(X, f)$ is a dynamical system.  If $f$ is $\mathcal{\bf I}$-transitive, then $\tau\subset \mathcal{\bf I}O(X)$.
\end{theorem}
\begin{proof}
  Let $f$ is $\mathcal{\bf I}$-transitive. Then by Theorem \ref{thm.tit.4.7}, $\mathcal{\bf I}\cap \tau=\{\varnothing\}$. Let $A\in\tau$, then by Theorem \ref{thm.tct.4.8.1}, $A\subset A^*$. This shows that $A\subset Int(A)\subset Int(A^*)$ and hence $A\in \mathcal{\bf I}O(X)$. Consequently, $\tau\subset \mathcal{\bf I}O(X)$. This completes the proof.
\end{proof}
\section{Ideal version of non-wandering point}\label{sec5}
In this section, we have introduced ideal version of non-wandering points called ideal-non-wandering points and investigate its various characterizations and properties.
\begin{definition}
Let $(X,f)$ be a dynamical system and $\mathcal{\bf I}$ be an ideal on $X$. A point $x\in X$ is called an ideal non-wandering (or simply $\mathcal{\bf I}$-non-wandering) if for every neighbourhood $M$ containing $x$, there exists a positive integer $n$ such that $f^{n}(M)\cap M\notin\mathcal{\bf I}$. The set of all ideal-non-wandering points is denoted by $\Omega_\mathcal{\bf I}(f)$.
\end{definition}
Every ``$\mathcal{\bf I}$-non-wandering point" is a ``non-wandering" point but the reverse may not be true.
For the existence of ideal non-wandering points and reverse inclusion, we consider the following examples:
\begin{example}
  \label{exm.tit.5.1}
  Consider an $\mathcal{\bf I}$-space $(X, \tau,\mathcal{\bf I})$  where $X=\{a, b, c\}$, $\tau=\{\varnothing,\{a\}, \{b\}, \{a,b\}, X\}$ and $\mathcal{\bf I}=\{\varnothing, \{c\}\}$. Take, $a\in X$. Then the neighbourhoods of $a$ are $\{a\}, \{a,b\}$ and $X$. Let us define a mapping $f:X \rightarrow X$ by $f(a)=b$, $f(b)=a$ and $f(c)=c$. Then for any  neighbourhood $M$ of $a$, $f^{n}(M)\cap M\notin\mathcal{\bf I}$ where $n=2$. Thus `$a$' is an $\mathcal{\bf I}$-non-wandering point of $X$.
\end{example}

\begin{example}
  \label{exm.tit.5.2}
  Consider an $\mathcal{\bf I}$-space $(X, \tau,\mathcal{\bf I})$  where $X=\{a, b, c\}$, $\tau=\{\varnothing,\{a\}, \{b\}, \{a,b\}, X\}$ and $\mathcal{\bf I}=\{\varnothing, \{a\}, \{b\}, \{a,b\}\}$. Take, $a\in X$. Then the neighbourhoods of $a$ are $\{a\}, \{a,b\}$ and $X$. Let us define a mapping $f:X \rightarrow X$ by $f(a)=b$, $f(b)=a$ and $f(c)=c$. Then for any  neighbourhood $M$ of $a$, $f^{n}(M)\cap M\neq\varnothing$  where $n=2$. Thus `$a$' is a non-wandering point of $X$. But for all $n\in\mathbb{N}$, $f^{n}(M)\cap M\in\mathcal{\bf I}$ when $M=\{a\}$ and $M=\{a,b\}$. Thus `$a$' is not an $\mathcal{\bf I}$-non-wandering point of $X$.
\end{example}

\begin{example}
  \label{exm.tit.5.3}
  Consider an $\mathcal{\bf I}$-space $(X, \tau,\mathcal{\bf I})$  where $X=\mathbb{N}$, set of all natural numbers, $\tau=\tau_{u}$, usual topology on $\mathbb{N}$ and $\mathcal{\bf I}=\mathcal{P}(\mathbb{N}\setminus \bf O)$, (power set of $(\mathbb{N}\setminus \bf O)$) where $\bf O$ is set of all odd natural numbers.  Let us define a mapping $f:X \rightarrow X$ by $f(k)=k$ for all $k\in X$. Then all odd natural numbers are non-wandering points of $\mathbb{N}$ but not $\mathcal{\bf I}$-non-wandering point of $\mathbb{N}$.
\end{example}
If a point $x\in X$ is a $\mathcal{\bf I}$-non-wandering point (i.e., $x\in\Omega_\mathcal{\bf I}(f)$), then no open set containing $x$ (or neighbourhood of $x$) is a member of $\mathcal{\bf I}$.\\
For the reverse inclusion, we consider the following:
\begin{theorem}
  Let $\mathcal{\bf I}$ be a codense ideal on a space $(X, \tau)$  and $(X, f)$ be a dynamical system. If $f$ is an open map, then any non-wandering point of $X$ is an $\mathcal{\bf I}$-non-wandering point.
\end{theorem}
\begin{proof}
Let $x\in X$ be a non-wandering point. Then for every non-empty open set $M$ of $x$, there exists a positive integer $k$ such that $f^{k}(M)\cap M\neq\varnothing$. Since $f$ is open, then $f^{k}(M)\cap M\in\tau$. This implies $f^{k}(M)\cap M\notin\mathcal{\bf I}$ since the ideal $\mathcal{\bf I}$ is codense. Thus, $x$ is an $\mathcal{\bf I}$-non-wandering point. This completes the proof.
\end{proof}
\begin{definition}
Let $(X,f)$ be a dynamical system and $\mathcal{\bf I}$ be an ideal on $X$. The dynamical system $(X,f)$ is said to have $\mathcal{\bf I}$-dense orbit with respect to the ideal $\mathcal{\bf I}$ if there exists a point $x\in X$ such that $A^{*}(\mathcal{\bf I})=X$ where $A$ is an  orbit of $x$.
\end{definition}
For the existence of $\mathcal{\bf I}$-dense orbit, we consider the following examples:
\begin{example}
  \label{exm.tit.5.6}
  Consider an $\mathcal{\bf I}$-space $(X, \tau,\mathcal{\bf I})$  where $X=\{a,b\}$, $\tau=\{\varnothing, X, \{a\}\}$,  and $\mathcal{\bf I}=\{\varnothing, \{b\}\}$.  Let us define a mapping $f:X \rightarrow X$ by $f(a)=f(b)=a$. Then the orbit of $b$ is $\mathcal{\bf I}$-dense.
\end{example}
\begin{example}
  \label{exm.tit.5.7}
  Suppose $X= \{0\}\cup\{\frac{1}{n} : n \in \mathbb{N}\}$ (where $\mathbb{N}$ denotes the set of all natural numbers) is a set endowed with the usual topology and consider an ideal $\mathcal{\bf I}=\{\varnothing, \{0\}\}$. Let us define a mapping $f:X \rightarrow X$ by $f(\frac{1}{n})=\frac{1}{(n+1)}$, $n=1, 2, 3,...$ and $f(0)=0$. Clearly $f$ is open and continuous. Then, the orbit of the point $x=1$ is $\mathcal{\bf I}$-dense.
\end{example}

\begin{theorem}[\cite{DGR1999}]
\label{thm.tit.4.8}
  For an $\mathcal{\bf I}$-space $(X,\tau,\mathcal{\bf I})$, the following statements are equivalent:\\
  $(1)$ $\mathcal{\bf I}$ is completely codense.\\
  $(2)$ Every dense set is $\mathcal{\bf I}$-dense.
\end{theorem}
This theorem is one of the rectification of the Remark $6.9(2)$ of \cite{LL2013}.

\begin{theorem}[\cite{DK2003}]
  \label{thm.tit.4.9}
  Let $X$ be a compact and Hausdorff space with a countable base and $f:X\rightarrow X$ be a continuous mapping. If $f$ is a topologically transitive, then it has a dense orbit.
\end{theorem}

\begin{proposition}
\label{prop.tit.02}
   Let $\mathcal{\bf I}$ be a completely codense ideal on a compact and Hausdorff space $X$  with a countable base and $f:X\rightarrow X$ be a continuous mapping. If $f$ is a $\mathcal{\bf I}$-transitive, then it has an $\mathcal{\bf I}$-dense orbit.
\end{proposition}
\begin{proof}
 Topologically ideal transitivity of $f$, implies $f$ is topologically transitive. Since $X$ is a compact and Hausdorff space with countable base and $f$ is a continuous mapping, then by Theorem \ref{thm.tit.4.9}, $f$ has a dense orbit. Also since $\mathcal{\bf I}$ is  completely codense ideal, then by Theorem \ref{thm.tit.4.8},  dense orbit of $f$ becomes $\mathcal{\bf I}$-dense orbit. Hence $f$ has an $\mathcal{\bf I}$-dense orbit. This completes the proof.
\end{proof}

 \begin{proposition}
\label{prop.tit.03}
   Let $\mathcal{\bf I}$ be an  ideal of nowhere dense subsets of a compact and Hausdorff space $X$  with a countable base and $f:X\rightarrow X$ be a continuous mapping. If $f$ is an $\mathcal{\bf I}$-transitive, then it has an $\mathcal{\bf I}$-dense orbit.
\end{proposition}
\begin{proof}
  The proof is straight forward and hence omitted.
\end{proof}

\begin{proposition}
\label{prop.tit.04}
   Let $\mathcal{\bf I}$ be a completely codense ideal on a compact and Hausdorff space $X$  with a countable base and $f:X\rightarrow X$ be a continuous and open mapping. Then, $f$ is an $\mathcal{\bf I}$-transitive if and only if $\Omega_\mathcal{\bf I}(f)=X$ and $f$ has  $\mathcal{\bf I}$-dense (respectively, $*$-dense) orbit.
\end{proposition}

\begin{proof}
 Suppose $f$ is an $\mathcal{\bf I}$-transitive. Clearly, it has an $\mathcal{\bf I}$-dense orbit, i.e., there exists $x_{0}\in X$ such that $O_{f}(x_{0})$ is $\mathcal{\bf I}$-dense in $X$. Given that $f$ is $\mathcal{\bf I}$-transitive, this implies that $\mathcal{\bf I}$ is codense. Further, $f$ is topological transitive, then for any neighbourhood $M$ of $x$, $f^{n}(M)\cap M\neq\varnothing$. This implies that $f^{n}(M)\cap M\notin \mathcal{\bf I}$ (since $f$ is open) and hence $\Omega_\mathcal{\bf I}(f)=X$.\\

Conversely, suppose $f$ has an $\mathcal{\bf I}$-dense orbit and $\Omega_{\mathcal{\bf I}}(f)=X$ and let $M$, $N$ be two non-empty open subsets of $X$. Let $x\in X$ have an $\mathcal{\bf I}$-dense orbit. Since every $\mathcal{\bf I}$-dense is dense, then the orbit of $x$ will enter both $M$ and $N$. Let $m$ and $n$ be the least positive integers such that $f^{m}(x)\in M$ and $f^{n}(x)\in N$. Assume $m<n$ and take $k=n-m$. Then, obviously $f^{k}(M)\cap N\neq \varnothing$ and hence $f^{k}(M)\cap N\notin\mathcal{\bf I}$, since $f$ is open and $\mathcal{\bf I}$ is a completely codense ideal. Thus, $f$ is an $\mathcal{\bf I}$-transitive.\\ This completes the proof.
\end{proof}

\begin{corollary}
\label{cor.tit.4.3}
   Let $\mathcal{\bf I}=\mathcal{N}$ be an  ideal of nowhere dense subsets of a compact and Hausdorff space $X$  with a countable base and $f:X\rightarrow X$ be a continuous and open mapping. Then, $f$ is an $\mathcal{\bf I}$-transitive if and only if $\Omega_\mathcal{\bf I}(f)=X$ and $f$ has  $\mathcal{\bf I}$-dense(respectively, $*$-dense)orbit.
\end{corollary}
 We are ending this write up with some further scope of this research. One can define $\mathcal{\bf I}$-transitive through generalized open sets viz semi-open, preopen set, $b$-open set and $\beta$-open set etc.  by the following way: A dynamical system $(X, f)$ is called  $K$-topologically ideal transitive (or simply $K$-$\mathcal{\bf I}$-transitive) if for every pair of non-empty $K$-open sets $A$ and $B$, there exists a positive integer $n$ such that $f^{n}(A)\cap B\notin\mathcal{\bf I}$. In this case, $K$ may be the semi-open set \cite{NL1963}, preopen set \cite{AMN1982}, $b$-open set \cite{DA1996}, $\beta$-open set \cite{AEM1983}. As we know open set implies semi-open set, then semi $\mathcal{\bf I}$-transitive $\implies$  $\mathcal{\bf I}$-transitive and thus, we conclude the  followings:\\

${\beta-\mathcal{\bf I}-\text{transitive}}\Rightarrow {b-\mathcal{\bf I}-\text{transitive}}\Rightarrow {\text{Semi}-\mathcal{\bf I}-\text{transitive}}
 \Rightarrow {{\bf I}-\text{transitive}}\Rightarrow \\
 \text{Topologically  transitive}$, and \\
 ${\beta-\mathcal{\bf I}-\text{transitive}}\Rightarrow {b-\mathcal{\bf I}-\text{transitive}}\Rightarrow {\text{Pre}-\mathcal{\bf I}-\text{transitive}}
 \Rightarrow {{\bf I}-\text{transitive}}\Rightarrow \\
 \text{Topologically  transitive}$\\

However, reverse inclusion of the above relation may not be true in general. For these, consider the following examples:

\begin{example}
  \label{exm.ncf.08}
  Suppose $X$ is an infinite set endowed with the indiscrete topology $\tau$ and $\mathcal{\bf I}$ is a proper ideal on $X$. Now $SO(X,\tau) = \{\varnothing, X\}$ and $BO(X, \tau) = \wp(X)$, the power set of $X$. Let us define a mapping $f:X \rightarrow X$ by $f(x)=x$, for all $x\in X$. Then, $f$ is semi-$\mathcal{\bf I}$-transitive but not $b$-$\mathcal{\bf I}$-transitive.
\end{example}

\begin{example}
  \label{exm.ncf.09}
  Suppose $X=\{a, b, c\}$ is set endowed with the  topology $\tau=\{\varnothing, X, \{a\},$ $ \{b\}, \{a, b\}\}$ and $\mathcal{\bf I}=\{\varnothing,\{c\}\}$. Now, $PO(X, \tau) = \{\varnothing, X, \{a\}, \{b\}, \{c\}, \{a, b\},\{a, c\},\{b, c\}\}$. Let us define a mapping $f:X \rightarrow X$ by $f(a)=b$, $f(b)=a$ and  $f(c)=c$. Then, $f$ is $\mathcal{\bf I}$-transitive but not Pre-$\mathcal{\bf I}$-transitive.
\end{example}

\begin{example}
  \label{exm.ncf.10}
  Suppose $X=\{a, b, c\}$ is set endowed with the  topology $\tau=\{\varnothing, X, \{a\},$ $ \{b\}, \{a, b\}\}$ and $\mathcal{\bf I}=\{\varnothing,\{c\}\}$. Now $PO(X, \tau) = \{\varnothing, X, \{a\}, \{b\}, \{a, b\}\}$ and $BO(X, \tau) = \{\varnothing, X, \{a\}, \{b\}, \{a, b\}, \{a, c\}, \{b, c\}\}$. Let us define a mapping $f:X \rightarrow X$ by $f(a)=b$, $f(b)=a$ and  $f(c)=c$. Then, $f$ is Pre-$\mathcal{\bf I}$-transitive but not $b$-$\mathcal{\bf I}$-transitive.
\end{example}

\begin{example}
  \label{exm.ncf.11}
  Suppose $X=\{a, b, c, d\}$ is set endowed with the  topology $\tau=\{\varnothing, X, \{a\},$ $ \{c\}, \{a, c\}, \{a, c, d\}\}$ and $\mathcal{\bf I}=\{\varnothing,\{b\}\}$. Now, $SO(X, \tau) = \{\varnothing, X, \{a\}, \{c\}, \{a, c\},  \{a, b\}, \{a, d\}, $ $\{b, c\}, \{c, d\}, \{a, b, c\}, \{a, b, d\}, \{a, c, d\}, \{b, c, d\}\}$. Let us define a mapping $f:X \rightarrow X$ by $f(a)=c$, $f(b)=b$, $f(c)=a$ and  $f(d)=d$. Then, $f$ is $\mathcal{\bf I}$-transitive but not Semi-$\mathcal{\bf I}$-transitive.
\end{example}

  \section*{Conclusion}
   Through this paper, we have introduced  equivalent definition of topologically transitive and a new type of transitivity called topologically ideal transitive (or $\mathcal{\bf I}$-transitive). These are  related to earlier transitivities. Topological ideal transitivity implies topologically transitive. But reverse may not be true. Our several transitivity measures weaker as well as strong transitivity of the earlier transitivities.

\end{document}